\documentclass[11pt]{amsart}
\usepackage[mathscr]{eucal}
\usepackage{palatino, mathpazo, amsfonts, mathrsfs, amscd}
\usepackage[all]{xy}
\usepackage{url}
\usepackage{amssymb, amsmath, amsthm}
\usepackage{stackrel}

\newtheorem{theorem}{Theorem}[section]
\newtheorem{lemma}[theorem]{Lemma}

\newtheorem{proposition}[theorem]{Proposition}
\newtheorem{corollary}[theorem]{Corollary}
\theoremstyle{remark}
\newtheorem{remark}[theorem]{Remark}
\newtheorem{example}[theorem]{Example}
\newtheorem{definition}[theorem]{Definition}

\newtheorem{conventions}[theorem]{Convention}
\numberwithin{equation}{subsection}

\newcommand{\on}{\operatorname}

\newcommand{\ol}{\overline}

\newcommand{\RR}{\mathbb{R}}
\newcommand{\CC}{\mathbb{C}}
\newcommand{\ZZ}{\mathbb{Z}}
\newcommand{\PP}{\mathbb{P}}
\newcommand{\QQ}{\mathbb{Q}}

\newcommand{\ii}{\mathbb{1}}
\newcommand{\inv}{^{-1}}
\newcommand{\tl}{\tilde}
\title%[]
{Birationality of Berglund--H\"ubsch--Krawitz Mirrors}

\author{Mark Shoemaker}
\address{Department of Mathematics, University of Michigan,
Ann Arbor, MI 48109-1043, U.S.A.}
\email{shoemama@umich.edu}
\date{28 September 2012}
\begin{document}

\begin{abstract}
We investigate a \emph{multiple mirror phenomenon} arising from Berglund--H\"ubsh--Krawitz mirror symmetry.  We prove that the different mirror Calabi--Yau orbifolds which arise in this context are in fact birational to one another.
%Given a Landau--Ginzberg singularity $(W,G)$ satisfying certain conditions, one may obtain a Calabi--Yau orbifold $[Z_W]$ as a hypersurface in a quotient of weighted projective space.  Berglund--H\"ubsh--Krawitz mirror symmetry 

%Given a hypersurface $Z_W$ defined by a Landau--Ginzberg model $(W,G)$, we prove that the different mirror orbifolds which occur by replacing the polynomial $W$ are birational to one another.  
%In the process we relate the  Berglund--H\"ubsh--Krawitz mirror procedure  to Batyrev--Borisov mirror symmetry.

\end{abstract}

\maketitle

\small
\tableofcontents
\normalsize
%\newpage

\setcounter{section}{-1}

%{\bf To Do \begin{enumerate}
%%\item Fix notational discrepancy between rays and ray generators
%%\item *****decide what to do about $(X_\Sigma, \mu_i \ldots )$ notation.  Here's an idea, how about rather than writing $(X_\Sigma, \{ \mu_0, \ldots, \mu_n\})$, just write $\{\sum_{i=0}^n \mu_i = 0\} \subseteq X_\Sigma$.  Well the problem is that we don't have a "mirror" procedure written out as explicitly now.  But that is probably just fine, the statement is now just that $\{\sum_{i=0}^n \nu_i = 0\} = \ZZ_{W^T}$
%%\item add explanation about orbifold being birational to its coarse space
%\item add intro...
%%\item add section comparing BHK to BB mirror symmetry
%\item bibliography
%\item change intro to question 2
%%\item Notational issues: \begin{enumerate}
%%%\item $[Z_W]$ vs $Z_W/\tl{G}]$, the problem being that the def of $G^T$ depends on $W$, so notation like $[Z_W/\tl{G^T}]$ is misleading when there are two different $W$'s.
%%%\item get rid of "invertible LG hypersurface"...
%%%\item $\tl{G^*}$.... speaks for itself
%%\end{enumerate}
%%\item Go through the main proof
%%\item check the fano thing for your corollary
%\end{enumerate}
%}

\section{Introduction}

Mirror symmetry has been a driving force in geometry and physics over the last twenty years.  Roughly speaking the mirror conjecture \cite{CK} predicts a deep relationship between certain pairs $(V, V^\circ)$ of Calabi--Yau $n$-folds.  At the level of cohomology for example, we have the isomorphism
\begin{equation}\label{rotate} H^{p,q}(V; \CC) \cong H^{n-p, q}(V^\circ; \CC).
\end{equation}
%Since the debut of mirror symmetry to mathematics, there have arisen many general techniques for constructing such mirror pairs.  We will focus specifically on
%%the Batyrev--Borisov (BB) mirror construction \cite{BB} and 
%the
%Berglund--H\"ubsh--Krawitz (BHK) mirror construction detailed in \cite{Kr}. 
%%where %in both cases
%\eqref{rotate} has been proven.  
The so called \emph{multiple mirror phenomenon} arises when a particular Calabi--Yau $V$ corresponds to more than one mirror $V^\circ$ via different mirror constructions.  In many cases there is no clear relationship between the different mirrors of $V$ which appear (see examples~\ref{ex1}, \ref{ex2} and~\ref{ex3}).  In this paper we investigate one such instance of the multiple mirror phenomenon coming from Berglund--H\"ubsh--Krawitz (BHK) mirror symmetry.  We prove that the different BHK mirrors are in fact birational. 
%and that BHK mirror symmetry coincides with BB mirror symmetry up to birational equivalence.

BHK mirror symmetry for invertible singularities together with the LG/CY state space correspondence of Chiodo--Ruan allows us to construct a wealth of examples of ``mirror pairs'' of Calabi--Yau orbifolds arising as hypersurfaces in quotients of weighted projective space.  Let $(W,G)$ be a Landau--Ginzberg (LG) singularity.  That is $W$ is a nondegenerate quasi-homogeneous polynomial satisfying 
 \[W(\lambda^{c_0} X_0, \ldots, \lambda^{c_n} X_n) = \lambda^d W(X_0, \ldots, X_n)\]
 for $\lambda \in \CC^*$, and $G$ is a subgroup of $\on{Aut}(W)$.  Assume $W$ satisfies the \emph{Calabi--Yau} condition \[\sum_{i=0}^n c_i = d.\] Assume that $G$ is a subgroup of $SL_{n+1}(\CC)$ and contains the distinguished element \[j_W = \left( \begin{array}{c}
 e^{2 \pi i c_0/d} \\
  \vdots \\
  e^{2 \pi i c_n/d}
 
 \end{array}\right),\] acting as $j_W: (X_0, \ldots, X_n) \mapsto (e^{2 \pi i c_0/d}X_0, \ldots, e^{2 \pi i c_n/d} X_n)$.
Then the group $\tl{G} := G/\langle j_W \rangle$ acts on the projective stack $[\PP_W] :=  [\PP(c_0, \ldots, c_n)]$ and we obtain a Calabi--Yau orbifold
 \[ [Z_W] := \{W = 0\} \subset [\PP_W/\tl{G}] .\]
 If we assume additionally that $W$ is an invertible polynomial, we can apply BHK mirror symmetry to obtain a mirror LG singularity $(W^T, G^T)$, and a new Calabi--Yau \[ [Z_{W^T}] = \{W^T = 0\} \subset [\PP_{W^T}/\tl{G^T}].\]   This gives the mirror pair $[Z_W]$ and $[Z_{W^T}]$.  By a theorem of Chiodo--Ruan \cite{CR}, the Hodge diamonds of $[Z_W]$ and $[Z_{W^T}]$ are related by a $90^\circ$ rotation as in \eqref{rotate}.  We say $(W,G)$ is of \emph{Calabi--Yau-type} if $W$ satisfies the CY condition and $\langle j_W \rangle \leq G \leq SL_{n+1}(\CC)$.
 
 Consider two invertible CY-type theories $(W, G)$ and $(W', G)$ such that the weights $(c_0, \ldots c_n)$ of $W$ coincide with the weights of $W'$ and $G \leq \on{Aut}(W) \cap \on{Aut}(W')$.
%such that $[\PP_{W}/\tl{G}] \cong [\PP_{W'}/\tl{G'}]$.  
Then the spaces $[Z_W]$ and $[Z_{W'}]$ are both Calabi--Yau hypersurfaces in $[\PP(c_0, \ldots, c_n)/\tl{G}]$.  They are related by a smooth deformation and so give two representatives of the same mirror family.  Thus the respective BHK mirrors $[Z_{W^T}]$ and $[Z_{{W'}^T}]$ give two different mirrors of the mirror family of $[Z_W]$.  The relationship between these two orbifolds has been unknown up until now.  In this paper we resolve this question.
\begin{theorem}[$=$ Theorem~\ref{t:main}]
Let $(W, G)$ and $(W', G)$be  invertible and of CY-type such that the weights of $W$ and $W'$ are $(c_0, \ldots, c_n)$ and $G \leq \on{Aut}(W) \cap \on{Aut}(W')$ (so both $[Z_W]$ and $[Z_{W'}]$ are hypersurfaces in $[\PP(c_0, \ldots, c_n)/\tl{G}]$).  Then  
the respective BHK mirrors $[Z_{W^T}]$ and $[Z_{{W'}^T}]$ are birational.
%Let 
%$[Z_W]$ and $[Z_{W'}]$ be two CY hypersurfaces in $[\PP(c_0, \ldots, c_n)/\tl{G}]$ corresponding to LG theories $(W, G)$ and $(W', G)$ such that the weights of $W$ and $W'$ are $(c_0, \ldots, c_n)$ and $G \leq \on{Aut}(W) \cap \on{Aut}(W')$.  Then  
%the respective BHK mirrors $[Z_{W^T}]$ and $[Z_{{W'}^T}]$ are birational.
\end{theorem}
H.\ Iritani first suggested that the above might be true.
 
 \begin{example}\label{ex1}
 Let \[{W_1} = X_0^5 + X_1^5 + X_2^5 + X_3^5 + X_4^5\] and let $G = \langle j_{W_1} \rangle$.  In this case $Z_{W_1}  = \{{W_1} = 0\} \subset \PP^4$ is the Fermat quintic threefold in $\PP^4$.  As is well known, in this case the BHK mirror coincides with the Batyrev--Borisov mirror (\cite{BB}).  ${W_1}^T = Y_0^5 + Y_1^5 + Y_2^5 + Y_3^5 + Y_4^5$ is again a Fermat polynomial and ${G_1}^T = SL_5(\CC) \cap \on{Aut}({W_1}) \cong (\ZZ/5\ZZ)^4.$  The mirror
 %space $[\PP_{{W_1}^T}/\tl{{G_1}^T}] = [\PP^4/(\ZZ/5\ZZ)^3]$, and 
 \[ [Z_{{W_1}^T}] = \{Y_0^5 + Y_1^5 + Y_2^5 + Y_3^5 + Y_4^5 = 0\} \subset [\PP^4/(\ZZ/5\ZZ)^3]\] is the well known \emph{mirror quintic} orbifold.
 \end{example}
 
 \begin{example}\label{ex2}
 Consider now the degree five homogeneous polynomial of \emph{chain type} described in Example 6 of \cite{CR}, \[{W_2} = X_1^4X_2 + X_2^4X_3 + X_3^4X_4 + X_4^4X_5 + X_5^5.\]  Again let ${G_2} = \langle j_{W_2} \rangle \;(=\langle j_{W_1} \rangle)$.  From the LG model $({W_2}, {G_2})$ we obtain another degree five hypersurface $Z_{{W_2}} = \{{W_2} = 0\} \subset \PP^4$. 
 %\[ [Z_{{W_2}}] = \{X_1^4X_2 + X_2^4X_3 + X_3^4X_4 + X_4^4X_5 + X_5^5 = 0\} \subset \PP^4.\]
 In this case however, \[{{W_2}}^T = Y_1^4 + Y_1Y_2^4+ Y_2Y_3^4 + Y_3Y_4^4 + Y_4Y_5^5\] is no longer homogeneous of degree five, but rather quasi-homogeneous:
 %weights $(c_0, \ldots, c_4; d) = (64, 48, 52, 51, 41; 256)$.
 \[{{W_2}}^T(\lambda^{64} Y_0, \lambda^{48}Y_1, \lambda^{52}Y_2, \lambda^{51}Y_3, \lambda^{41} Y_4) = \lambda^{256} {{W_2}}^T(Y_0, \ldots, Y_4).\]  One can check that in this case ${{G_2}}^T = \langle j_{{{W_2}}^T}\rangle$, thus we obtain a hypersurface in weighted projective space \[ [Z_{{{W_2}}^T}] = \{{{W_2}}^T = 0\} \subset [\PP(64, 48, 52, 51, 41)].\] 
 \end{example}
 
  \begin{example}\label{ex3}
Next consider the polynomial of mixed type \[{W_3} = X_1^4X_2 + X_2^5 + X_3^5 + X_4^5 + X_5^5.\]  Again letting ${G_3} = \langle j_{W_3} \rangle \;(=\langle j_{W_1} \rangle)$ we obtain a degree five hypersurface $Z_{{W_3}} = \{{W_3} = 0\} \subset \PP^4$. 
 %\[ [Z_{{W_3}}] = \{X_1^4X_2 + X_2^4X_3 + X_3^4X_4 + X_4^4X_5 + X_5^5 = 0\} \subset \PP^4.\]
 In this case, \[{{W_3}}^T = Y_1^4 + Y_1Y_2^5+ Y_3^5 + Y_4^5 + Y_5^5\] is quasi-homogeneous:
 %weights $(c_0, \ldots, c_4; d) = (64, 48, 52, 51, 41; 256)$.
 \[{{W_3}}^T(\lambda^{5} Y_0, \lambda^{3}Y_1, \lambda^{4}Y_2, \lambda^{4}Y_3, \lambda^{4} Y_4) = \lambda^{20} {{W_3}}^T(Y_0, \ldots, Y_4),\]  and ${{G_3}}^T = SL_5(\CC) \cap \on{Aut}(W_3) = \langle j_{{W_3}^T}, g_1, g_2 \rangle$, where \[ j_{{W_3}^T} = \left(\begin{array}{c}
 e^{2 \pi i ( 1/4) } \\
 e^{2 \pi i (3/20)} \\
 e^{2 \pi i (1/5)} \\
 e^{2 \pi i (1/5)} \\
 e^{2 \pi i (1/5)} \end{array} \right), \quad g_1 = \left(\begin{array}{c}
 1 \\
 1 \\
 e^{2 \pi i (1/5)} \\
 e^{2 \pi i (4/5)} \\
 1 \end{array} \right),  \quad g_2 = \left(\begin{array}{c}
1 \\
 1 \\
 1\\
 e^{2 \pi i (1/5)} \\
 e^{2 \pi i (4/5)} \end{array} \right).
 \]
 %The group $\tl{{G_3}^T} \cong (\ZZ/5\ZZ)^2$, thus
 In this case the mirror is a hypersurface in the quotient of weighted projective space \[ [Z_{{{W_3}}^T}] = \{{{W_3}}^T = 0\} \subset [\PP(5,3,4,4,4)/(\ZZ/5\ZZ)^2].\] 
 \end{example}

 In the above examples, the hypersurfaces $Z_{W_1}$, $Z_{W_2}$, and $Z_{W_3}$ are related by smooth deformations, and should therefore be considered as three different representatives of the same mirror family.  However we see that there is no obvious relationship between the respective mirrors $[Z_{{W_1}^T}]$, $[Z_{{W_2}^T}]$, and $[Z_{{W_3}^T}]$.

 %Let us phrase the question more generally.
 
 %A question naturally arises from this discussion.  
 %The first concerns the so-called \emph{multiple mirror phenomenon}.  
 
%Secondly, it is unclear how the above mirror construction relates to BB mirror symmetry.  To be more specific, given an LG model $(W,G)$, such that the (coarse moduli) space $\PP_W/\tl{G}$ is Gorenstein and Fano, we can apply BHK mirror symmetry to obtain 
%$[Z_{W^T}]$ as a mirror of $[Z_W]$.  On the other hand, if we let $V$ denote a crepant resolution $[Z_W]$, 
%BB mirror symmetry gives a mirror $V^\circ$.  One may ask what is the relationship between $[Z_{W^T}]$ and $V^\circ$.\footnote{The case of $\PP_W = \PP^{n}$ is the content of Theorem $5.1.1$ in \cite{Ba}.  
%%The general result will follow from the more general considerations of section \ref{toricBHK}
%}  As a corollary to the above theorem, we obtain
%\begin{theorem}[$=$ Corollary~\ref{c:cor}]
%Given an LG model $(W, G)$ such that $\PP_W/\tl{G}$ is Gorenstein, the BHK mirror $[Z_{W^T}]$ is birational to a Fermat representative of the BB mirror family.
%\end{theorem}

\subsection{Organization of the paper}
In section~\ref{s:background} we provide a brief introduction to %BB and 
BHK mirror symmetry and set notation for what follows.
In section~\ref{toricBHK} we relate $[Z_W]$ to $[Z_{W^T}]$ via a toric construction.
% reminiscent of the Batyrev mirror construction.  
 In section~\ref{theorems} we prove the birationality of the multiple mirrors.

\subsection{Acknowledgements}  The author would like to thank his advisor, Y.~Ruan for his tremendous help and guidance over the years.  He thanks P.~Clarke for first presenting the mirror construction of section~\ref{toricBHK} at the University of Michigan RTG Mirror Symmetry Workshop in the Spring of 2012.  He would also like to thank L.~Borisov, W.~Fulton, M.~Musta\c{t}\v{a}, and M.~Satriano for sharing their expertise.  The author was partially supported by NSF RTG grant DMS-0602191.

\section{Berglund--H\"ubsch--Krawitz mirror symmetry}\label{s:background}

%\subsection{Batyrev--Borisov mirror symmetry}
%
%The most well-known formulation of mirror symmetry, Batyrev--Borisov (BB) mirror symmetry, relates Calabi--Yau complete intersections in Gorenstein Fano toric varieties \cite{BB}.  For our purposes it will be enough to simply recall Batyrev's original construction, which can be viewed as a specialization to the particular case of toric hypersurfaces \cite{Ba}.  Namely, consider a lattice $N \cong \ZZ^{n+1}$ and $M = \on{Hom}(N; \ZZ)$.  Let $\Delta$ be a reflexive polytope in $M_\RR$.  Let $\Sigma$ denote a choice of \emph{maximal projective crepant partial (MPCP) desingularization} of $\Delta$, that is, $\Sigma$ is a maximal projective simplicial refinement of the normal fan of $\Delta$ satisfying $\Sigma(1) \subset \Delta^\circ - \{0\}$.
%Consider the toric variety $X_\Sigma$ refining $\PP_{\Delta}$.  $X_\Sigma$ is a Gorenstein Fano variety, thus its anticanonical bundle $-K_{\Sigma}$ is ample.  The zero set of a generic section of $-K_{\Sigma}$ defines a Calabi--Yau variety $V $ in $ X_\Sigma$.  
%
%The polar dual of $\Delta$, $\Delta^\circ \subset N_\RR$ is again a reflexive polytope.  Let $\Sigma^\circ $ denote an MPCP desingularization of $\Delta^\circ$.  Let $V^\circ $ in $ X_{\Sigma^\circ}$ be the zero set of a generic section of $-K_{\Sigma^\circ}$.  Then $V$ and $V^\circ$ give a Batyrev mirror pair.  
%
%\subsection{Berglund-H\"ubsh-Krawitz mirror symmetry}
\begin{conventions} \label{conv:1}
We work in the algebraic category.
The term \emph{orbifold} means ``smooth separated Deligne--Mumford stack
of finite type over $\mathbb{C}$.''

In our notation, stacks will always be written in brackets as $[X]$, and $X$ without brackets will denote the coarse underlying space.
\end{conventions}
Berglund-H\"ubsh-Krawitz (BHK) mirror symmetry is %another formulation %of mirror symmetry 
defined in terms of Landau-Ginzberg theories.  A Landau-Ginzberg (LG) theory is defined by a pair $(W, G)$ where $W = W(X_0, \ldots, X_n)$ is a nondegenerate quasi-homogeneous polynomial, and $G$ is a subgroup of $\on{Aut}(W)$, the group of diagonal symmetries of $W$.  By quasi-homogeneity we mean that there exist integers $c_0, \ldots, c_n$ such that for $\lambda \in \CC^*$, \[W(\lambda^{c_0} X_0, \ldots, \lambda^{c_n} X_n) = \lambda^d W(X_0, \ldots, X_n).\] We assume that $\gcd(c_0, \ldots, c_n)=1$.  $W$ is then quasi-homogeneous of degree $d$, with \emph{weights} $c_i$. We will also sometimes refer to the rational numbers $q_i = c_i/d$ as the \emph{fractional weights} of $W$.

BHK mirror symmetry applies when $W$ is a so-called \emph{invertible} polynomial, meaning the number of monomials is equal to the number of variables.  We can write \[W  = \sum_{i=0}^n \prod_{j=0}^n X_j^{e_{ij}}.\]  Let $E = \left(e_{ij}\right)$ denote the exponent matrix of $W$.  Let $E\inv = \left(e^{ij}\right)$.  %The columns $\left( e^{i1}, \ldots, e^{in} \right)$ 
Let $\rho_j$ denote the diagonal matrix whose $i^{th}$ diagonal entry is $\exp(2 \pi i e^{ij})$.  The group $\on{Aut}(W)$ is equal to $\langle \rho_0, \ldots, \rho_n\rangle$.  There is a distinguished element $j_W$ whose $i^{th}$ diagonal entry is $\exp( 2 \pi i q_i)$.  In fact $j_W = \rho_0 \cdots \rho_n$.  Given $G \leq \on{Aut}(W)$, the pair $(W, G)$ gives an invertible LG model.

The transpose polynomial $W^T$ is defined by transposing the exponent matrix $E$.  Namely
\[W^T := \sum_{j=0}^n \prod_{i=0}^n Y_i^{e_{ij}}.\]  As above, $\on{Aut}(W^T) = \langle \ol{\rho_0}, \ldots, \ol{\rho_n} \rangle$ where $\ol{\rho_i}$ is the diagonal matrix whose $j^{th}$ diagonal entry is $\exp( 2 \pi i e^{ij})$.  Define \[ G^T := \left\{ \prod_{i=0}^n \ol{\rho_i}^{s_i} \Bigg| \prod_{i=0}^n X_i^{s_i} \text{ is $G$-invariant} \right\}.\] 
The mirror LG model is given by $(W^T, G^T)$.  See~\cite{BH, Kr} for more information.

We may associate to $(W, G)$ a nondegenerate bigraded vector space, the \emph{Fan-Jarvis-Ruan-Witten state space} 
\[H_{FJRW}^*(W, G; \CC) = \bigoplus_{p,q} H_{FJRW}^{p,q}(W, G; \CC),\] 
defined by orbifolding the classical space of Lefschetz thimbles \cite{FJRW}.  The mirror theorem for BHK mirror symmetry then states that the Hodge diamond of $H_{FJRW}^*(W^T, G^T; \CC)$ is obtained from that of $H_{FJRW}^*(W, G; \CC)$ by performing a $90^\circ$ rotation \cite{Kr}.

A connection between Landau--Ginzberg theories and Calabi--Yau complete intersections is given by Chiodo and Ruan in \cite{CR}.  
They show that for a special class of LG models $(W, G)$, the state space $H_{FJRW}^*(W, G; \CC)$ is isomorphic to that of a Calabi--Yau hypersurface in (a quotient of) weighted projective space. Given a polynomial 
\[W  = \sum_{i=0}^m \prod_{j=0}^n X_j^{e_{ij}}\] and a group $G \leq \on{Aut}(W)$, assume that $G$ contains the distinguished element $j_W$.  Assume further that the weights $(c_0, \ldots , c_n)$ satisfy the 
\emph{Calabi--Yau condition} \[\sum_{i=0}^n c_i = d.\]  Let $[\PP_W]$ denote the weighted projective stack $ [\PP(c_0,\ldots, c_n)] = [(\CC^{n+1} - \{0\})/\CC^*]$ with $\CC^*$-weights $(c_0,\ldots, c_n)$.  
%Let $[Z_W] = \{W = 0\}$ the Calabi--Yau hypersurface in $[\PP_W]$.  
Define $\tl{G} := G/\langle j_W \rangle$.  $\tl{G}$ acts on $[\PP_W]$ and preserves the polynomial $W$.  Thus we can define the hypersurface \[[Z_W] := \{W = 0\} \subset [\PP_W/\tl{G}].\]

The theorem below relates the state space of $(W,G)$ to the Chen-Ruan cohomology of $[Z_W]$.

\begin{theorem}[LG/CY correspondence for hypersurfaces, \cite{CR}]\label{LGCY}

Given a Landau-Ginzberg theory $(W, G)$ such that $\sum_{i=0}^n c_i = d$ and $j_W \in G$, there exists an isomorphism of bigraded vector spaces \[H_{FJRW}^{p,q}(W, G; \CC) \cong H^{p,q}_{CR}( [Z_W]; \CC ).\]

\end{theorem}
%This motivates the following definition. \begin{definition}
%A Landau-Ginzberg model $(W,G)$ satisfying the conditions of Theorem \ref{LGCY} will be called a \emph{Landau-Ginzberg (LG) hypersurface}. 
%\end{definition}  
%
Given $(W, G)$ as above, let us now assume that $W$ is an invertible polynomial.  One can show that $G \leq SL_{n+1}(\CC)$ if and only if $G^T$ contains ${j_{W^T}}$, the distinguished element of $W^T$.  Furthermore if $W$ satisfies the Calabi--Yau condition $W^T$ will also.  So in this case $(W^T, G^T)$ will also correspond to a hypersurface \[[Z_{W^T}] = \{W^T = 0\} \subset [\PP_{W^T}/\tl{G^T}].\]  The hypersurfaces $[Z_W]$ and $[Z_{W^T}]$ give a pair of mirror Calabi--Yau orbifolds.
Applying Theorem~\ref{LGCY} plus BHK mirror symmetry, we see that the standard relationship between the Hodge diamonds of mirror pairs holds:
%\[H^{p,q}_{CR}( [Z_W]; \CC) \cong H_{FJRW}^{p,q}(W, G; \CC) \cong H_{FJRW}^{n-1-p,q}(W^T, G^T; \CC) \cong H^{n-1-p,q}_{CR}( [Z_{W^T}]; \CC),\] 
\[
\begin{array}{ccc}
H^{p,q}_{CR}( [Z_W]; \CC)& & H^{n-1-p,q}_{CR}( [Z_{W^T}]; \CC)\\
\downarrow \cong & & \uparrow \cong \\
H_{FJRW}^{p,q}(W, G; \CC) &\xrightarrow{\cong} & H_{FJRW}^{n-1-p,q}(W^T, G^T; \CC)  
\end{array}.
\]

\begin{definition}
In this paper we deal only with Landau-Ginzberg theories $(W,G)$ of \emph{Calabi--Yau-type}, that is: 
\begin{enumerate}

\item[i.] the weights $(c_0, \ldots, c_n)$ satisfy the Calabi--Yau condition $\sum_{i=0}^n c_i = d$,
\item[ii.] $\langle j_W\rangle \leq G \leq SL_{n+1}(\CC)$.
\end{enumerate}
\end{definition}

\begin{conventions} Henceforth, we will always assume that our polynomial $W$ is invertible.

We will use the term \emph{BHK mirror symmetry} to refer both to the correspondence between the pairs $(W,G)$ and $(W^T, G^T)$, as well as the induced correspondence between $[Z_W]$ and $[Z_{W^T}]$.
\end{conventions}

%what is the relationship between the mirror symmetry just described relating the hypersurfaces $[Z_W]$ and $[Z_{W^T}]$ and BB mirror symmetry for hypersurfaces?  To be more specific, consider an LG model $(W,G)$ such that the $c_i$ divides $d$ for all $i$.  This is equivalent to requiring that $\PP_W = \PP(c_0, \ldots, c_n)$ be Gorenstein.  The condition that $G \leq SL_{n+1}(\CC)$ implies that $\PP_W/\tl{G}$ is also Gorenstein.  $\PP_W/\tl{G}$ is therefore a Gorenstein Fano toric variety, and can therefore be expressed as $\PP_{\Delta}$ for some reflexive polytope $\Delta$ in a lattice $M$.  Let $\Sigma$ be a projective subdivision of $\Sigma$ and let $f: X_\Sigma \to \PP_\Delta$ denote the corresponding resolution.  Then the pullback of $Z_W$ via $f$ defines a Calabi--Yau $V$ in $X_\Sigma$.  Then one can ask how the mirror $[Z_{W^T}]$ of $[Z_W]$ relates to $V' \subset X_{\Sigma '}$, the BB mirror of $V$.  A partial result is already known.  Namely, in the case where $W$ is a Fermat polynomial and $G = \langle j \rangle$, $\PP_{W^T}/\tl{G^T}$ is isomorphic to $\PP_{\Delta^\circ}$, and the pullback of $Z_{W^T}$ under $g: X_{\Sigma ' } \to \PP_{\Delta}$  coincides with a representative $V'$ in $X_{\Sigma ' }$.

\section{Rephrasing BHK mirror symmetry}\label{toricBHK}

%At their hearts, both Batyrev--Borisov and BHK mirror symmetry reduce  complex geometric relationships to simple combinatorial constructions.  In the case of BB mirror symmetry, this relies on constructing the dual polytope.  In the case of BHK mirror symmetry, one uses the transpose of the exponent matrix and its inverse to construct the transpose polynomial and transpose group.  We would like to relate these two combinatorial constructions.  This first involves expressing BHK mirror symmetry, at least on the level of the course spaces, in toric terms.  
BHK mirror symmetry reduces complex geometric relationships to a simple combinatorial construction based on transposing the exponent matrix and its inverse.  Using a construction due to Clarke, we can reformulate BHK mirror symmetry, at least on the level of coarse moduli spaces, into the language of toric varieties.  

\begin{remark}
In \cite{Cl}, Clarke gives a general construction which specializes to that detailed below.  
\end{remark}

\begin{remark}
In \cite{Bo}, Borisov gives a different way of rephrasing BHK mirror symmetry in toric terms.
\end{remark}

\subsection{Expressing $Z_W$ in toric language}\label{s:2.1}
Let $(W, G)$ be of CY-type.  Label the $i^{th}$ monomial of $W$ as $Y_i$, so  $W =  \sum_{i=0}^n Y_i = \sum_{i=0}^n \prod_{j=0}^n X_j^{e_{ij}}$.  We can view $X_0, \ldots, X_n$ as homogeneous coordinates on the toric variety $\PP_W/\tl{G}$, the coarse underlying space of the quotient $[\PP_W/\tl{G}]$.  Then $\{W = 0\}$ defines the hypersurface $Z_W$.

%.  Assume the weights $q_0, \ldots , q_n$ satisfy the Calabi--Yau condition $\sum_i q_i = 1$.  Let $G \subset \on{Aut}(W) \cap SL_{n+1}(\CC)$ be a subgroup containing the element $j = \exp(2 \pi i q_0, \ldots, 2 \pi i  q_n).$  Then the Landau-Ginzberg model $(W,G)$ satisfies the conditions of Theorem \ref{LGCY}.    In particular, we can view $X_0, \ldots, X_n$ as homogeneous coordinates on the toric variety $\PP_W/\tl{G} = \PP(d q_0,\ldots, d q_n)/\tl{G}$, and $\{W = 0\}$ defines a Calabi--Yau hypersurface.  In \cite{CR} the correspondence is proven. 

Note first that $\PP_W/\tl{G}$ is a normal toric variety, and can therefore by expressed as $X_\Sigma$ for $\Sigma$ a fan in a lattice $N \cong \ZZ^n$.  Because $X_\Sigma$ is a finite quotient of weighted projective space, the fan $\Sigma$ takes a specific form.  The following two lemmas will be useful
in describing such fans and their properties.

%In particular if $X_\Sigma$ is $n$-dimensional, $\Sigma(1)$ will consist of $n +1$ rays $\rho_0, \ldots , \rho_n$.  

First, it will be helpful in what follows to obtain an explicit description of the short exact sequence  %(see section 3 of~\cite{CK}) 
\begin{equation}\label{SES} 0 \to M \to \ZZ^{\Sigma(1)} \to A_{n-1}(X_\Sigma) \to 0\end{equation} for the divisor class group in the particular case of $X_\Sigma \cong \PP_W/\tl{G}$.

\begin{lemma}\label{setup1}
Let  $(W, G)$ be of CY-type and let $(c_0, \ldots, c_n)$ be the weights of $W$.  Then there exists a fan $\Sigma$ such that the corresponding toric variety $X_\Sigma \cong \PP_W/\tl{G}$, and the short exact sequence \eqref{SES} may be written as \[ 0 \to M \to \ZZ^{n + 1} \xrightarrow{\phi \oplus \tl{\phi}}  \ZZ \oplus \tl{G} \to 0,\] where $\phi$ is given by $(a_0, \ldots, a_n) \mapsto (c_0 a_0 + \cdots + c_n a_n)$.

\end{lemma}

\begin{proof}

We will begin by constructing a fan $\Sigma ' $ in a lattice $N' \cong \ZZ^{n}$ such that  $X_{\Sigma'} \cong \PP_W$.
%We will first compute the case where $\tl{G}$ is trivial ($G = \langle j \rangle$).  
Define the lattice $N' := \ZZ^{n+1}/(c_0, \ldots, c_n) \ZZ$ and let $M' : = \on{Hom}(N', \ZZ)$.  Let $e_j$ denote the standard basis in $\ZZ^{n+1}$ and let $\nu_j$ denote the image of $e_j$ in $N'$.  The fan $\Sigma'$ for $\PP(c_0, \ldots, c_n) = \PP_W$ consists of cones generated by proper subsets of $\{\nu_0, \ldots, \nu_n\}$.  Let $\rho_j := \RR_{\geq 0} \nu_j$ denote the associated ray in $N'_\RR$.

The relation $\sum_{j = 0}^n c_j \cdot \nu_j = 0$ holds in $N'$, thus we have a short exact sequence \begin{equation}\label{SESWP} 0 \to M' \to \ZZ^{n + 1} \xrightarrow{\phi}  \ZZ \to 0\end{equation} where $\phi$ is as above, and the map $M' \to \ZZ^{n+1}$ is given by 
\[m \mapsto ( \langle m, \nu_0 \rangle, \ldots, \langle m , \nu_n \rangle ).\]  To show that this short exact sequence is in fact the sequence \eqref{SES} for the divisor class group of $\PP_W$, we must prove that $\nu_j$ is the generator of $\rho_j$ for each $j$ or, equivalently, that each $\nu_j$ is primitive.  This is not always the case for weighted projective space, but it does in fact hold for the weights which arise in our context.  

It is a simple exercise 
%which we will leave to the reader 
to show that $\nu_j$ can be written as $k \cdot v$ for some $v \in N'$ if and only if $k |
\gcd( c_0, \ldots, \widehat{c_j},
%c_{i-1}, c_{i + 1}, 
\ldots, c_n).$  Now, assume to the contrary that for $\PP_W$ there exists a $j$ such that $\nu_j$ is not primitive.  Let $k = \gcd( c_0, \ldots, \widehat{c_j}, \ldots, c_n)$.  Then by the above $k >1$.  By assumption $\gcd( c_0, \ldots,  c_n) = 1$, so $k$ does not divide $c_j$, and therefore does not divide $d = \sum_{j = 0}^n c_j$.  But since $k$ divides $c_i$ for all $i \neq j$, every degree $d$ monomial in the homogeneous variables $X_0, \ldots, X_n$ must contain a factor of $X_j$.  This would imply that $X_j$ divides $W$, contradicting our assumption that $W$ is nondegenerate.  Therefore each $\nu_j$ is the generator of $\rho_j$, and \eqref{SESWP} is in fact the same sequence as \eqref{SES} for the fan $\Sigma '$.  This proves the lemma when $\tl{G}$ is trivial.

In the general case, we can construct the fan $\Sigma$ by first constructing $\Sigma' \subset N'$ as above and then embedding $N'$ in an appropriate suplattice $N$ such that $N/N' = \tl{G}$.  Let $\Sigma$ denote the image of $\Sigma'$ in $N$, it follows from standard toric geometry arguments that $X_{\Sigma } \cong \PP_W/\tl{G}$.  

We claim that $\nu_0, \ldots, \nu_n$ are still primitive in $N$.  Recall that points in $M'$ correspond to rational functions on $X_{\Sigma '}$, specifically, the point $m \in M'$ gives the rational function $\prod_{j=0}^n X_j^{\langle m, \nu_j\rangle}$.  Global sections of the anticanonical sheaf are given by rational functions in $X_0, \ldots, X_n$ with at most simple poles at $X_j = 0$ for $0 \leq j \leq n$.   The Calabi--Yau condition above guarantees that for each $i$, $Y_i/(X_0 \cdots X_n) = \prod_{j=0}^n X_j^{e_{ij}}/(X_0 \cdots X_n)$ is a rational function on $\PP_W$, and can thus be represented by a point $\mu_i \in M'$.

By assumption, $G \leq SL_{n+1}(\CC)$, thus $X_0 \cdots X_n$ is invariant under the action of $G$.  $Y_i$ is also $G$-invariant, so the rational function $Y_i/(X_0 \cdots X_n)$ descends to a function on $X_{\Sigma  }$.  Equivalently, for each $i$, $\mu_i$ is in $M = \on{Hom}(N, \ZZ)$.  Now assume that for some $j$, $\nu_j$ can be written as a multiple $\nu_j = k \cdot \tl{\nu_j} $ for some $\tl{\nu_j}$ in $N$.
By the nondegeneracy of $W$, there exists a monomial $Y_i$ in $W$ with no factors of $X_j$.  Thus $Y_i/(X_0 \cdots X_n)$ will have a factor of $X_j\inv$.  Thus $k \langle \mu_i, \tl{\nu_j}\rangle = \langle \mu_i, k\tl{\nu_j}\rangle = \langle \mu_i, \nu_j \rangle = -1$.  Therefore $k$ must be $\pm 1$.  Thus $\nu_j$ is primitive in $N$ for all $j$.

The above discussion implies that we can write \eqref{SES} for $X_\Sigma = \PP_W/\tl{G}$ as \[ 0 \to M \to \ZZ^{n + 1} \xrightarrow{\phi \oplus \tl{\phi}}  \ZZ \oplus H \to 0,\] where $\phi$ is given by $(a_0, \ldots, a_n) \mapsto (c_0 a_0 + \cdots + c_n a_n)$ and $H$ is a finite group.  We can identify $M'$ with the kernel of $\phi$, which results in the exact sequence \[ 0 \to M \to M' \xrightarrow{ \tl{\phi}}  H \to 0.\]
Applying $\on{Hom}( - , \CC^*)$ to the above sequence, we deduce $H \cong \ker( T_{N'} \to T_N) \cong N/N' = \tl{G}$.

\end{proof}

As a partial converse, we can explicitly describe the structure of $X_\Sigma$ as a quotient of weighted projective space when $\Sigma$ takes a certain specific form.  

%\begin{lemma}
%Let $\nu_0, \ldots, \nu_n$ be primitive elements of a lattice $N$.  Assume that the sequence \[ 0 \to M \xrightarrow{A} \ZZ^{n + 1} \xrightarrow{\phi \oplus \tl{\phi}}  \ZZ \oplus \tl{G} \to 0,\]
%
%is exact, where $A: m \mapsto (\langle m, \nu_0 \rangle, \ldots, \langle m, \nu_n \rangle)$ and  $\phi$ is given by ${\displaystyle (a_0, \ldots, a_n) \mapsto (c_0 a_0 + \cdots + c_n a_n)}$ ($c_i > 0$ for all $i$).  Let $N' \subset N$ denote the lattice generated by $\langle \nu_0, \ldots, \nu_n \rangle$.  Define the fan $\Sigma ' \subset N '$ by taking cones over proper faces of $\on{conv}(\nu_0, \ldots, \nu_n)$.  Let $\Sigma$ be the image of $\Sigma '$ in $N$.
%Then $X_{\Sigma '} \cong \PP(c_0, \ldots, c_n) $, and $X_{\Sigma} \cong \PP(c_0, \ldots, c_n)/(N/N') $ where $ N/N' \cong \tl{G}$.
%\end{lemma}
\begin{lemma}\label{setup2}
Let $\nu_0, \ldots, \nu_n$ be primitive elements of a lattice $N$.  Assume that the sequence \[ 0 \to M \xrightarrow{A} \ZZ^{n + 1} \xrightarrow{\phi \oplus \tl{\phi}}  \ZZ \oplus \tl{G} \to 0,\]
is exact, where $A: m \mapsto (\langle m, \nu_0 \rangle, \ldots, \langle m, \nu_n \rangle)$ and  $\phi: (a_0, \ldots, a_n) \mapsto (c_0 a_0 + \cdots + c_n a_n)$ ($c_i > 0$ for all $i$).  Define the fan $\Sigma  \subset N $ by taking cones over proper faces of $\on{conv}(\nu_0, \ldots, \nu_n)$.  Then $X_{\Sigma} \cong \PP(c_0, \ldots, c_n)/\tl{G}$.
\end{lemma}

%\begin{remark}
%Notice that by the arguments in the above proof, if $\Sigma \subset N$ is a fan such that $X_\Sigma \cong \PP_W/\tl{G}$ and the SES \eqref{SES} is of the form of the lemma, then letting $N' = \langle \nu_0 \ldots \nu_n \rangle$, $N/N' \cong \tl{G}$.
%\end{remark}
\begin{proof}
First let $N' \subset N$ denote the lattice generated by $\langle \nu_0, \ldots, \nu_n \rangle$.  
Define the fan $\Sigma ' \subset N '$ by taking cones over proper faces of $\on{conv}(\nu_0, \ldots, \nu_n)$.  Note that the inclusion $N' \rightarrow N$ maps $\Sigma '$ to $\Sigma$.  Letting $M' = \on{Hom}(N'; \ZZ)$, the sequence 
\[ 0 \to M' \to \ZZ^{n + 1} \xrightarrow{\phi}  \ZZ \to 0\]
is exact, where $\phi$ is as above.  Thus $X_{\Sigma '}$ is isomorphic to $\PP(c_0, \ldots, c_n)$.  By standard toric arguments, $X_{\Sigma} \cong X_{\Sigma '}/(N/N')$.  As in the proof of the previous lemma, the short exact sequence
\[ 0 \to M \to M' \xrightarrow{ \tl{\phi}}  \tl{G} \to 0\]
is exact, and after applying $\on{Hom}( - , \CC^*)$ to the above sequence, we deduce $N/N' \cong \ker( T_{N'} \to T_N) \cong  \tl{G}$.

\end{proof}

Assume we have a hypersurface $Z_W \subset \PP_W/\tl{G}$ corresponding to $(W,G)$.
Let $\Sigma \subset N$ be the fan constructed in lemma \ref{setup1} such that $X_\Sigma \cong \PP_W/\tl{G}$.  Let $M = \on{Hom}(N; \ZZ) $ denote the dual lattice.  Let
$\{\nu_0, \ldots, \nu_n\}  \subset N$ denote the primitive elements corresponding to the generators of $\Sigma(1)$.
%$\rho_0, \ldots, \rho_n$. 
Label them such that $X_j$ is the homogeneous coordinate corresponding to $\nu_j$.
As in the proof of lemma \ref{setup1}, each monomial $Y_i$ in $W$ corresponds to a rational function $Y_i / X_0 \cdots X_n$ on $X_\Sigma$, and therefore to a point $\mu_i \in M$.
The hypersurface $Z_W$ is defined by the vanishing of the section $\sum_{i = 0}^n \mu_i$ of the anticanonical sheaf. 
 
We can thus express the hypersurface $Z_W = \{W = 0\} \subset \PP_W/\tl{G}$ 
%rephrase the data of an LG model $(W,G)$ 
in toric language, as 
 a toric variety $X_\Sigma \cong \PP_W/\tl{G}$ plus $n+1$ points $\mu_0, \ldots, \mu_n \in M$ such that $\langle \mu_i, \nu_j\rangle + 1 = e_{ij}$.  We assemble this information in the form of a proposition as it is crucial to what follows.

\begin{proposition}\label{p:2.4}
Let $(W,G)$ be of CY-type with  $[Z_W] \subset [\PP_W/\tl{G}]$ the associated hypersurface.  There exists a fan $\Sigma \subset N$ plus $n + 1$ points $\mu_0, \ldots, \mu_n \in M$ such that $X_\Sigma \cong \PP_W/\tl{G}$, and $Z_W$ is defined by the vanishing of the section $\sum_{i = 0}^n \mu_i$ of $-K_\Sigma$.  Thus the coarse space $Z_W$ is determined completely by the toric variety $X_\Sigma$ and the choice of sections $\mu_0, \ldots, \mu_n \in M$.

Furthermore, the construction of the fan $\Sigma \subset N$ as well as the lattice $M$ depend only on the weights of $W$ and the choice of group $G$.
\end{proposition}

\begin{remark}\label{r:toric language}  
We will denote the space $Z_W$ as constructed above by \[\bigg\{\sum_{i = 0}^n \mu_i = 0\bigg\} \subseteq X_\Sigma\] with the implicit assumption that the rational function $\sum_{i = 0}^n \mu_i $ is viewed as a section of $-K_\Sigma$.
%$\left(X_\Sigma, \{\mu_0, \ldots, \mu_n\}\right)$.
\end{remark}

\subsection{The BHK mirror construction using toric varieties}\label{s:2.2}

Applying BHK mirror symmetry to $(W,G)$ yields the pair $(W^T, G^T)$, which, again applying Theorem \ref{LGCY}, corresponds to a hypersurface $Z_{W^T} = \{W^T = 0\} \subset \PP_{W^T}/\tl{G^T}$.  In fact we can relate $Z_W$ to $Z_{W^T}$ via a toric construction from \cite{Cl} reminiscent of Batyrev's original mirror construction \cite{Ba}.
%We now describe a toric mirror construction which corresponds with BHK mirror symmetry.  Assume we are given an LG model $(W,G)$.  
As in the above remark, let $\{\sum_{i=0}^n \mu_i = 0\} \subseteq X_\Sigma$ be the toric data associated to $(W,G)$. We will first construct a dual fan $\Sigma^\vee \subset M$.  

The nondegeneracy of $W$ implies that $\{\mu_0, \ldots, \mu_n\} \subset M$ form the vertices of an $n$-dimensional simplex in $M$.  The nondegeneracy of $W^T$ guarantees each $\mu_i$ is primitive.  

\begin{definition}\label{d:dual}
Given $(W,G)$ of CY-type defining a hypersurface $\{\sum_{i=0}^n \mu_i = 0\} \subseteq X_\Sigma$, let $\Sigma^\vee \subset M$ denote the fan whose cones consist of the cones over proper faces of $\on{conv}(\mu_0, \ldots, \mu_n)$.  
\end{definition}

By lemma \ref{setup2}, $X_{\Sigma^\vee}$ is again isomorphic to a quotient of weighted projective space by a finite group.  

%In other words, the maximal cones $\{\sig_i\}_{i=0}^n = \Sigma^\vee(n)$ are given by $\sig_i = \{ \sum_{k = 0}^n a_k m_k \big| a_i = 0, a_k \geq 0 \text{ for } k \neq i\}$.  The set $\sig_i$ together with their faces form $\Sigma^\vee$. One can check (or see below) that $X_{\Sigma^\vee}$ 
We write the homogeneous coordinate ring of $X_{\Sigma^\vee}$ as $\CC[Y_0, \ldots, Y_n]$, where $Y_i$ is the coordinate associated to $\mu_i$.  Note that each $\nu_j$ now corresponds to a rational function on $X_{\Sigma^\vee}$.  Furthermore, we know $\langle \mu_i, \nu_j\rangle \geq -1$, so each $\nu_j$ gives a section of $-K_{\Sigma^\vee}$.
In fact the homogeneous function corresponding to the section $\sum_{j=0}^n \nu_j $ is exactly $W^T$:  
\begin{equation}\label{e:transpose}
 W^T =\sum_{j=0}^n\prod_{i = 0}^n Y_i^{\langle \mu_i, \nu_j \rangle + 1},\end{equation}
thus $W^T$ gives a quasi-homogeneous function on $X_{\Sigma^\vee}$ whose zero section defines a Calabi--Yau hypersurface.  This gives a ``mirror'' Calabi--Yau: \[\{W^T = 0\} \subset X_{\Sigma^\vee}.\]

\begin{remark}
Note that $X_{\Sigma^\vee}$, and therefore also our hypersurface, depend upon choices of $\mu_0, \ldots, \mu_n$, which are determined by our choice of quasi-homogeneous polynomial $W.$  In the case where $X_\Sigma = \PP_W/\tl{G}$ is a Gorenstein Fano variety and $W$ is a Fermat polynomial, this construction coincides with the Batyrev mirror construction.  More precisely, if $\Delta$ is a polytope in $M$ such that $\Sigma$ is the normal fan of $\Delta$, then in the above situation $\Sigma^\vee$ gives the normal fan of $\Delta^\circ$.
\end{remark}

We claim that this construction agrees with the BHK mirror construction.  Namely we will show that $X_{\Sigma^\vee} \cong \PP_{W^T}/\tl{G^T},$ and thus the 
%LG model $(W^T, G^T)$ 
hypersurface $Z_{W^T}$ coincides with $\{\sum_{j=0}^n \nu_j = 0\} \subseteq X_{\Sigma^\vee}$.

\begin{proposition}\label{p:ambient}
Given $(W,G)$ of CY-type, construct $\Sigma$ as in Lemma~\ref{setup1} such that $X_\Sigma \cong \PP_W/\tl{G}$, and define $\Sigma^\vee$ as in definition~\ref{d:dual}.  Then $X_{\Sigma^\vee} \cong \PP_{W^T}/\tl{G^T}.$
\end{proposition}

\begin{proof}
Recall that $\{\nu_0, \ldots, \nu_n\} = \Sigma(1)$ and $\{\mu_0, \ldots, \mu_n\} \subset M$ denote the rational functions corresponding to the monomials $\{Y_0, \ldots, Y_n\}$.  We have the following short exact sequences: \begin{eqnarray} &0 \to M \xrightarrow{A} \ZZ^{\Sigma(1)} \to A_{n-1}(X_\Sigma) \to 0, \label{SESA} \\
&0 \to N \xrightarrow{B} \ZZ^{\Sigma^\vee(1)} \to A_{n-1}(X_{\Sigma^\vee}) \to 0. \label{SESB}
\end{eqnarray}

By lemma \ref{setup1} $A_{n-1}(X_\Sigma) \cong \ZZ \oplus \tl{G}$ and we can rewrite \eqref{SESA} as \begin{equation}
0 \to M \xrightarrow{A} \ZZ^{\Sigma(1)} \xrightarrow{\phi \oplus \tl{\phi}} \ZZ \oplus \tl{G} \to 0, \label{SESA2}
\end{equation} where $\phi(a_0, \ldots, a_n) = dq_0 a_0 + \cdots + d q_n a_n.$

%$X_\Sigma \cong \left((\CC^*)^{n+1} - 0 \right)/ \on{Hom}(A_{n-1}(X_\Sigma))$, where $\on{Hom}(A_{n-1}(X_\Sigma)) \cong \CC^* \times \tilde{G}$ with $\CC^*$ acting with weights $dq_1, \ldots , dq_n$.  We see that $A_{n-1}(X_\Sigma) \cong \ZZ \oplus \tilde{G}$.  

Note that $BA^T$ gives the pairing matrix between $\{\nu_0, \ldots, \nu_n\}$ and $\{\mu_0, \ldots, \mu_n\}$, \[ \left(BA^T\right)_{ij} = \langle \mu_i, \nu_j\rangle.\]  Let $\ii$ denote the $(n+1) \times (n+1)$ matrix with all entries equal to one.  Then the exponent matrix $E$ for $W$ can be expressed as $E = BA^T + \ii$.  Now, the (fractional) weights of $W^T$ are given by \[ \vec{p} =\left(\begin{array}{c} p_0 \\ \vdots \\ p_n \end{array}\right) = \left(E^T\right)\inv \left(\begin{array}{c} 1 \\ \vdots \\ 1 \end{array}\right). \]  
Recall that $\PP_{W^T} = \PP(\ol{d} p_1, \ldots, \ol{d}p_n)$, where $\ol{d}$ is the smallest integer such that $\ol{d}p_i$ is an integer for all $i$ and $\gcd(\ol{d}p_1, \ldots, \ol{d}p_n) = 1$.  We will show first that $X_{\Sigma^\vee}$ is a quotient of $\PP(\ol{d}p_1, \ldots, \ol{d}p_n).$  To see this, consider the map $\psi: \ZZ^{\Sigma^\vee} \to \ZZ$ given by $ \psi(b_0, \ldots , a_n) = \ol{d}p_0b_0 + \cdots \ol{d}p_nb_n.$  We claim that $B(N) \subseteq \ker(\psi).$ This can be rephrased as saying that $\ol{d}\vec{p}^T B = \vec{0}^T,$ which holds if and only if $AB^T \ol{d}\vec{p} = \vec{0}$ since $A$ has full rank.  By the Calabi--Yau condition for $W^T$ we know that $\sum_{i = 0}^n p_i = 1$, so the above statement is equivalent to the condition $E^T \vec{p} = AB^T \vec{p} + \ii \vec{p} = \vec{0} + \vec{1}= \vec{1}.$ But this is immediate as $E^T \vec{p} = E^T \cdot \left(E^T\right)\inv \vec{1} = \vec{1}.$ Because $\gcd(\ol{d}p_1, \ldots, \ol{d}p_n) = 1$, $\on{ker}(\psi)$ contains all linear relations in $B(N)$.  Thus we can rewrite \eqref{SESB} as 
\begin{equation}
0 \to N \xrightarrow{B} \ZZ^{\Sigma^\vee(1)} \xrightarrow{\psi \oplus \tl{\psi}} \ZZ \oplus H \to 0, \label{SESB2}
\end{equation} where $H $ is a finite group.  By lemma \ref{setup2}, this implies that $X_{\Sigma^\vee} \cong \PP_{W^T}/H$.  More precisely, the proof of lemma \ref{setup2} shows that $X_{\Sigma^\vee} = \PP_{W^T}/(M/\langle \mu_0, \ldots, \mu_n\rangle)$ and we can identify 
$H$ with $M/ \langle \mu_0, \ldots, \mu_n\rangle$.

To complete the proof of the proposition, we will now identify $M/\langle \mu_0, \ldots, \mu_n \rangle$ with $\tl{G^T}$.  Recall that the group $G^T$ can be defined as \[ G^T = \left\{ \prod_{i=0}^n \ol{\rho_i}^{s_i} \Bigg| \prod_{i=0}^n X_i^{s_i} \text{ is $G$-invariant} \right\}.\]  Given a $G$-invariant monomial $\prod_{i=0}^n X_i^{s_i}$, its degree can be expressed as $d \cdot (s_0, \ldots, s_n) E\inv \vec{1} \in d \ZZ$.  By the Calabi--Yau condition, $d \cdot (1, \ldots, 1) E\inv \vec{1} = d \sum_{i = 0}^n q_i = d$, so there exists some $k \in \ZZ$ such that $d \cdot (s_0 + k, \ldots, s_n + k) E\inv \vec{1} =0$ or in other words, such that $\prod_{i=0}^n X_i^{s_i + k}$ is degree zero.  Multiplying $\prod_{i=0}^n X_i^{s_i}$ by $\prod_{i=0}^n X_i^{k}$ corresponds to multiplying $\prod_{i=0}^n \ol{\rho_i}^{s_i}$ by $\left(j^T\right)^k$.  Therefore $\tl{G^T}$ is generated by elements of the form $\prod_{i=0}^n \ol{\rho_i}^{s_i}$ where $\prod_{i=0}^n X_i^{s_i}$ is degree 0.

Given an $n+1$-tuple $(s_0, \ldots, s_n)$ such that $\prod_{i=0}^n X_i^{s_i}$ is $G$-invariant and $\on{deg}\left(\prod_{i=0}^n X_i^{s_i}\right) = 0$,  the group element $\prod_{i=0}^n \ol{\rho_i}^{s_i}$ corresponds to the identity in $\tl{G^T}$ if, after multiplying by a multiple of $j_{W^T}$, it acts trivially on each coordinate, i.e., if \[(s_0, \ldots, s_n) E\inv + k \cdot (1, \ldots, 1) E\inv \in \ZZ^{n+1}\] for some $k \in \ZZ$.  Equivalently, $(s_0, \ldots, s_n)$ represents the trivial element if there exists a $k \in \ZZ$ such that \[(s_0, \ldots, s_n) + k(1, \ldots, 1) = (a_0, \ldots, a_n) E\] for some $a_0, \ldots, a_n \in \ZZ$.  
In this case \[\big((s_0, \ldots, s_n) + k(1, \ldots, 1)\big)E\inv \vec{1} = (a_0, \ldots, a_n) \vec{1} = \sum_{i = 0}^n a_i,\] 
and by assumption $(s_0, \ldots, s_n)E\inv \vec{1} = 0$, thus $k = k\left(1, \ldots, 1\right)E\inv \vec{1} = \sum_{i = 0}^n a_i$.  Since $(a_0, \ldots, a_n) E = (a_0, \ldots, a_n)\left( BA^T + \ii\right) = (a_0, \ldots, a_n)BA^T + \left(\sum_{i=0}^n a_i \right)(1, \ldots, 1)$, it follows that $\vec{s}^T = (s_0, \ldots, s_n)$ corresponds to the identity in $\tl{G^T}$ 
 if and only if \[ \left(\begin{array}{c} s_0 \\ \vdots \\ s_n \end{array}\right) \in \on{im}(AB^T). \] Finally, by \eqref{SESA} and the fact that $G \leq SL_{n+1}(\CC)$, 
 %{ \bf((((and the fact that $G \in SL_{n+1}$ - we need this to say that invt monomials with all positive coefficients correspond to invt degree 0 monomials by multiplying by $X_0 \cdot X_n$ this condition says that the 2 notions are equivalent, because it implies that $X_0 \cdot X_n$ is invariant under $G$)))) }, 
 the vectors $\vec{s}$ corresponding to degree zero monomials on $\PP_W$ preserved by $G$ are given by $\on{im}(A)$ in $\ZZ^{\Sigma(1)}$.  Thus $\tl{G^T}$ is identified with $\on{im}(A)/\on{im}(AB^T) \cong M/\on{im}(B^T) = M/\langle \mu_0, \ldots, \mu_n \rangle.$
 
 To see that both $H$ and $\tl{G^T}$ act in the same way on $\PP_{W^T}$, consider an element $m \in M$ representing an element of $H \cong M/ \langle \mu_0, \ldots, \mu_n\rangle$.  We can choose some 
 $\vec{r}$ in $\QQ^{n+1}$ 
 such that $B^T \vec{r} = m$.  
 Then in coordinates $m$ acts on the $Y_i$ coordinate as $e^{2 \pi i r_i}$.  Alternatively, if we view $m$ as a representative of an element of $\tl{G^T}$, then its coordinate-wise action is given by $e^{2 \pi i t_i}$ where $\vec{t} = (E^T)\inv A m$.   
 Now note that 
 \[\begin{split}
 \vec{t} &= (E^T)\inv A m 
 \\ &= (E^T)\inv A B^T \vec{r} 
 \\ &= (E^T)\inv A B^T \vec{r} + (E^T)\inv \ii \vec{r} - (E^T)\inv \ii \vec{r} 
 \\ &= (E^T)\inv E^T \vec{r}  - (E^T)\inv (\sum_{i=0}^n{r_i}) \vec{1} 
 \\ &= \vec{r} - (\sum_{i=0}^n{r_i}) \vec{p}.
 \end{split}
 \]  Thus $\vec{t}$ and $\vec{r}$ differ 
 %in their coordinate-wise action 
 by a (rational) multiple of $\ol{d}\vec{p}$.  
 But $\PP_{W^T}$ can be expressed as a quotient as $\{\CC^{n+1} - 0\}/\CC^*$ where the $\CC^*$ weights are given by $\ol{d} \vec{p}$, thus the action of $\vec{t}$ and $\vec{r}$ is the same on $\PP_{W^T}$.
 
\end{proof}

\begin{corollary}
The hypersurface $Z_{W^T}$ is isomorphic to $\{\sum_{j=0}^n \nu_j =0\} \subseteq X_{\Sigma^\vee}$.
\end{corollary}
\begin{proof}
By proposition~\ref{p:ambient}, $\PP_{W^T}/\tl{G^T} \cong X_{\Sigma^\vee}$.  Furthermore, by \eqref{e:transpose}, the zero set of the homogeneous function $W^T$ is exactly the vanishing locus of $\sum_{j=0}^n \nu_j$.
\end{proof}

%So in the case that $(W,G)$ is an LG model, we can rephrase BHK mirror symmetry entirely in terms of toric varieties.  First rephrase $(W, G)$ in toric language as $\{ \sum_{i=0}^n \mu_i = 0\}\subseteq X_\Sigma$ where $X_\Sigma \cong \PP_W/\tl{G}$ and $\mu_0, \ldots, \mu_n \in M$ are $n+1$ sections of the anticanonical bundle corresponding to the monomials of $W$.  Then the mirror of $\{ \sum_{i=0}^n \mu_i = 0\}\subseteq X_\Sigma$ is $\{\sum_{j=0}^n \nu_j =0\} \subseteq X_{\Sigma^\vee}$ where $\Sigma^\vee$ is defined in \ref{d:dual}, and $\nu_0, \ldots , \nu_n \in N$ are the primitive generators of the cones in $\Sigma(1)$.  This viewpoint will be useful in comparing the different BHK mirrors coming from different choices of polynomial.
%Namely, express $\PP_W/\tl{G}$ as a toric variety $X_\Sigma$, for $\Sigma$ a fan in $N$, and let $\mu_0, \ldots, \mu_n \in M = \on{Hom}(M; \ZZ)$ be the $n+1$ sections of the anticanonical bundle corresponding to the monomials of $W$.

\section{Multiple mirrors}\label{theorems}
%and the relation to BB mirror symmetry}

%Consider two LG models $(W, G)$ and $(W', G)$ such that the weights of $W$ coincide with the weights of $W'$ and $G \leq \on{Aut}(W) \cap \on{Aut}(W')$.  
%%Then $[Z_W]$ and $[Z_{W'}]$ give two different hypersurfaces in the same quotient of weighted projective space.  They are therefore equivalent up to a smooth deformation, and should be viewed as two representatives of the same mirror family.
%%$\PP_{W}/\tl{G} \cong \PP_{W'}/\tl{G'}$.  
%%Then the corresponding hypersurfaces are related by deformation.  
%Then the respective BHK mirrors $[Z_{W^T}]$ and $[Z_{{W'}^T}]$ give two different mirrors of the mirror family $[Z_W]$.  This is an example the so-called multiple mirror phenomenon.  In this section we will prove that these mirrors are in fact birational.

\begin{theorem}\label{t:main}
%Let $[Z_W]$ and $[Z_{W'}]$ be two CY hypersurfaces in $[\PP(c_0, \ldots, c_n)/\tl{G}]$ corresponding to CY singularities 
%
Given $(W, G)$ and $(W', G)$ of CY-type such that the weights of $W$ and $W'$ are $(c_0, \ldots, c_n)$ and $G \leq \on{Aut}(W) \cap \on{Aut}(W')$ (so both $[Z_W]$ and $[Z_{W'}]$ are hypersurfaces in $[\PP(c_0, \ldots, c_n)/\tl{G}]$), then  
the respective BHK mirrors $[Z_{W^T}]$ and $[Z_{{W'}^T}]$ are birational.
\end{theorem}

\begin{proof}
Note that given a CY-type $(W, G)$, the corresponding hypersurface $[Z_W]$ is always an effective orbifold.  In fact the set of points with nontrivial orbifold structure is a Zariski-closed subset of the space itself, thus the orbifold $[Z_W]$ is always birational to the coarse space $Z_W$.  Therefore, to prove the birationality of $[Z_{W^T}]$ and $[Z_{{W'}^T}]$ it suffices to work on the level of coarse spaces.  The theorem will follow once we show that $Z_{W^T}$ and $Z_{{W'}^T}$ are birational.

In the toric language of remark~\ref{r:toric language}, we can represent the hypersurfaces $Z_W$ and $Z_{W'}$ as $\{ \sum_{i=0}^n \mu_i = 0\}\subseteq X_\Sigma$ and $\{ \sum_{i=0}^n \mu_i ' = 0\}\subseteq X_\Sigma$ respectively.  Let $\Sigma^\vee$ denote the dual fan from definition~\ref{d:dual} with respect to $(W,G)$, so ${\Sigma^\vee}(1) = \{\mu_0, \ldots, \mu_n\}$. 
Let ${\Sigma '}^\vee$ denote the fan from definition~\ref{d:dual} with respect to $(W', G)$, so ${\Sigma '}^\vee (1) = \{\mu_0 ', \ldots, \mu_n '\}$.  Both fans are supported in $M$ ($M$ is determined only by the weights $(c_0, \ldots, c_n)$ and the choice of group $G$, see Proposition~\ref{p:2.4}).  The corresponding mirrors are given by 
\begin{equation}\label{e:mirrors}\begin{split}
Z_{W^T} &= \bigg\{\sum_{j=0}^n \nu_j =0\bigg\} \subseteq X_{\Sigma^\vee} \text{ and }
\\Z_{{W'}^T} &= \bigg\{\sum_{j=0}^n \nu_j =0\bigg\} \subseteq X_{{\Sigma '}^\vee}.
\end{split}
\end{equation}  

The toric varieties $X_{\Sigma^\vee}$ and $X_{{\Sigma '}^\vee}$ give two different compactifications of $T_M  = M \otimes \CC^*$.  By \eqref{e:mirrors} both $Z_{W^T}$ and $Z_{{W'}^T}$ contain $\{\sum_{j=0}^n \nu_j =0\} \subset T_M$ as an open subset.  This proves the claim.
%By \eqref{e:mirrors} we see that the vanishing locus of the section $\sum_{j=0}^n \nu_j $ in $ \Gamma(-K_{T_M})$ is contained as an open subset in
%both $Z_{W^T}$ and $Z_{{W'}^T}$. 
%This proves the claim.
\end{proof}

\end{document}